\documentclass[11pt]{amsart}
\usepackage{hyperref}
\usepackage{color}

\newtheorem{theorem}{Theorem}[section]
\newtheorem{proposition}[theorem]{Proposition}
\newtheorem{lemma}[theorem]{Lemma}
\newtheorem{corollary}[theorem]{Corollary}
\theoremstyle{definition}

\theoremstyle{remark}
\newtheorem{remark}[theorem]{Remark}

\newcommand{\U}{\bar{u}}

\newcommand{\R}{\mathbb{R}}

\numberwithin{equation}{section}
\DeclareMathOperator{\cexp}{c-exp}

\begin{document}

\title[Partial $W^{2,p}$ regularity]{\textbf{Partial $W^{2,p}$ regularity\\ for optimal transport maps}}

\author{Shibing Chen, Alessio Figalli}
\address{Mathematical Science Institute,
The Australian National University,
Canberra  ACT  2601}
\email{chenshibing1982@hotmail.com}

\address{Department of Mathematics RLM 8.100, The University of Texas at Austin, 2515 Speedway Stop C1200, Austin
TX 78712, USA}
\email{figalli@math.utexas.edu}
\thanks{}


\date{\today}

\dedicatory{}

 \begin{abstract}
We prove that, in the optimal transportation problem with general costs and positive continuous densities, the potential function is always of class $W^{2,p}_{loc}$ for any $p \geq 1$ outside of a closed singular set of measure zero. We also establish global $W^{2,p}$ estimates when the cost is a small perturbation of the quadratic cost. The latter result is new even when the cost
is exactly the quadratic cost.
 \end{abstract}

 \maketitle
\section{Introduction}

Regularity of optimal transport maps is a very important problem that has been studied extensively in the recent years. For the special case when 
the cost function is given by $c(x,y)=\frac{1}{2}|x-y|^2$ (or equivalently $c(x,y)=-x\cdot y$, see the discussion in \cite[Section 3.1]{DFreview}), Caffarelli \cite{C90, C902, C91, C92, C96} developed a deep regularity theory. However, for general costs functions the situation was much more complicated.
A major breakthrough 
happened in 2005 when Ma, Trudinger, and Wang \cite{MTW} introduced a fourth order condition on the cost function (now known as {\em MTW condition}) that
guarantees the smoothness of optimal transport map under suitable global assumptions on the data. Later, it was shown by Loeper \cite{L1} that the MTW condition is actually a necessary condition. Motivated by these results,
a lot of efforts have been devoted to understanding the regularity properties of optimal map under the MTW condition, see for instance  \cite{FL,Liu,TW1,TW2,FR,Loe2,LV,LTW, LTW2, FRV-surfaces,KMC1,FRV-nec,FRV,FKM,FKM-spheres}.

\vskip 10pt

Unfortunately, as observed by Loeper in \cite{L1}
and further noticed in many subsequent works, the MTW condition is extremely restrictive and many interesting costs do not satisfy this condition. Hence, a natural and important question became the following: What can we say about the regularity of optimal transport maps when the MTW condition fails?
A first major answer was given by De Philippis and Figalli \cite{DF}: there, the authors proved that, 
without assuming neither the MTW condition nor any convexity on the domains, for the optimal transport problem with positive continuous (resp. positive smooth)  densities, the  potential function is always $C_{loc}^{1,\alpha}$ (resp. smooth) outside a closed singular set of measure zero.
In a related direction,  Caffarelli, Gonz\'{a}les, and Nguyen \cite{CGN} obtained an interior $C_{loc}^{2,\alpha}$ regularity result of optimal transport problem when
the densities are $C^\alpha$ and
the cost function is of the form $c(x,y)=\frac{1}{p}|x-y|^p$ with $2<p<2+\epsilon$ for some $\epsilon \ll 1$ (or, $p>1$ and the distance between source and target is sufficiently large). This interior regularity result was later extended by us
to a global one \cite{CF}. 

\vskip 10pt

The aim of this work is to further develop the techniques 
introduced in \cite{CF, CF1, DF} and prove a partial $W^{2,p}$ regularity result. More precisely we show that, 
 for the optimal transport problem with positive continuous densities,
 there exists a closed singular set of measure zero
 outside which the  potential function is of class $W^{2,p}_{loc}$ for any $p>1$ (in particular, the singular set is {\em independent} of the exponent $p$).
 As a corollary of our techniques together with an argument due to Savin \cite{Savin}, we are able to obtain global $W^{2,p}$
 estimates when the domains are convex and the cost function is $C^2$-close to $-x\cdot y$. 
\vskip 10pt

The paper is organized as follows. In section 2 we introduce some notation and state our main
results. Then, in section 3 we prove our key Proposition \ref{p1}, and finally in the last section we prove our main results.

\section{Preliminaries and main results}

 First, we introduce some conditions which should be satisfied by the cost. Let $X$  and $Y$ be
 two  bounded open subsets of $\mathbb{R}^n$. 
 \medskip
  
 (\textbf{C0}) The cost function $c:X\times Y\to \R$ is of class $C^3$, with $\|c\|_{C^3(X\times Y)}<\infty.$
 
 (\textbf{C1}) For any $x\in X$, the map $Y \ni y\mapsto D_xc(x,y)\in \mathbb{R}^n$ is injective.
 
 (\textbf{C2}) For any $y\in Y$, the map $X \ni x\mapsto D_yc(x,y)\in \mathbb{R}^n$ is injective.
   
 (\textbf{C3}) $\det(D_{xy}c)(x,y)\neq 0$ for all $(x,y)\in X\times Y.$
 
 \medskip
 
 A function $u: X\rightarrow \mathbb{R}$ is said {\it $c$-convex} if it can be written as
 \begin{equation} \label{121}
 u(x)=\underset{y\in Y}{\sup}\{-c(x,y)+\lambda_y\}
 \end{equation}
 for some family of constants $\{\lambda_y\}_{y \in Y}\subset \R$.
 Note that  $(\textbf{C0})$ and \eqref{121} imply that a $c$-convex function is semiconvex, namely, there exists a constant $K$ depending only
 on $\|c\|_{C^2(X\times Y)}$ such that $u+K|x|^2$ is convex. One immediate consequence of the semiconvexity is that $u$ is twice differentiable almost everywhere.
 
 Thanks to  $(\textbf{C0})$ and  $(\textbf{C1})$ it is well known (see for instance \cite[Chapter 10]{V1}) that there exists a unique optimal transport map.
 Also, there exists a $c$-convex function $u$ such that
 the optimal map is a.e. uniquely characterized in terms of $u$ 
 (and for this reason we denote it by $T_u$) via the relation
\begin{equation}
\label{eq:Tu}
 -D_xc(x, T_u (x))= \nabla u(x) \qquad \text{for a.e. }x.
\end{equation}
As explained for instance in \cite[Section 2]{DF} (see also \cite{DFreview}), the transport condition
$(T_u)_\#f=g$ implies that $u$ solves at almost every point the Monge-Amp\`ere type equation
\begin{equation}
\label{eq:MA T}
\det\Bigl(D^2u(x)+D_{xx}c\bigl(x,\cexp_x(\nabla u(x))\bigr) \Bigr)=\left|\det\left(D_{xy}c\bigl(x,\cexp_x(\nabla u(x))\bigr) \right) \right| \frac{f(x)}{g(\cexp_x(\nabla u(x)))},
\end{equation}
where $\cexp$ denotes the \textit{$c$-exponential map} defined as
\begin{equation}
\label{eq:cexp}
\text{for any $x\in X$, $y \in Y$, $p \in \R^n$},
\qquad
\cexp_x(p)=y \quad \Leftrightarrow \quad p=-D_xc(x,y).
\end{equation}
Notice that, with this notation, $T_u(x)= \cexp_x(\nabla u(x))$.

 For a $c$-convex function, in analogy with the subdifferential for convex functions, we can talk about its 
 $c$-subdifferential: If $u: X\rightarrow \mathbb{R}$ is a $c$-convex function, the {\it $c$-subdifferential} of $u$ at $x$ 
 is the (nonempty) set
 $$\partial_cu(x):=\bigl\{y\in \overline{Y}: u(z)\geq -c(z,y)+c(x,y)+u(x)\qquad \forall\,z\in X  \bigr\}.$$
 We also define \emph{Frechet subdifferential} of $u$ at $x$ as
 $$\partial^-u(x):=\bigl\{p\in\mathbb{R}^n: u(z)\geq u(x)+p\cdot(z-x)+o(|z-x|)\bigr\}.$$
 It is easy to check that \begin{equation}
 \label{eq:partial c}
 y\in \partial_cu(x)\quad\Longrightarrow\quad -D_xc(x,y)\in \partial^-u(x).
 \end{equation}
Also, it is a well-known fact (see for instance \cite[Chapter 10]{V1}) that the transport map $T_u$
and the $c$-subdifferential $\partial_cu$ are related by the inclusion $$
T_u(x) \in \partial_cu(x).
$$
In particular, since $\partial_cu(x)$
is a singleton at every differentiability point of $u$
(this follows by \eqref{eq:partial c}), we deduce that
\begin{equation}
\label{eq:u diff}
\partial_cu(x)=\{T_u(x)\}\qquad \text{whenever $u$ is differentiable at $x$.}
\end{equation}
The analogue of sublevels of a convex functions is played by the {\em sections}: given $y_0 \in \partial_cu(x_0)$, we define
$$S(x_0, y_0, u, h):=\{x: u(x)\leq -c(x, y_0)+c(x_0, y_0)+c(x_0, y_0)+u(x_0)+h\}.$$
Note that, whenever $u$ is differentiable at $x_0$ then $y_0=T_u(x_0)$.
To simplicity the notation, we will use $S_h(x_0)$ to denote $S(x_0, y_0, u, h)$ when no confusion arises.

Finally, we recall that given $u$ $c$-convex,
its $c$-transform $u^c$ is defined as
$$
u^c(y):=\sup_{x \in X}\{-c(x,y)-u(x)\}.
$$
With this definition, $u^c$ plays the role of $u$
for the transportation problem from $g$ to $f$.\\ 

 Our first main result states that, if $f$ and $g$ are positive continuous densities, then $u$ is of class $W^{2,p}_{loc}$
 for any $p\geq 1$ outside a closed set of measure zero.
 A crucial fact in our proof is to show that the singular set $\Sigma$ is {\em independent} of $p$.
 
 \begin{theorem}\label{t1}
Let $u$ be the potential function for the optimal transport problem from $(X, f)$  to $(Y, g)$ with cost $c$ satisfying $(\textbf{C0})$-$(\textbf{C3})$. Suppose $f:X\to \R^+$ and $g:Y\to \R^+$ are positive continuous densities.
Then there exists a closed set $\Sigma\subset X$ of measure zero such that $u\in W^{2,p}_{loc}(X\setminus\Sigma)$ for any $p\geq 1.$
\end{theorem}

By a localization argument, the above theorem yields the following:

\begin{corollary}\label{c1}
Let $(M, \mathcal{G})$ be a smooth closed Riemannian manifold,
and denote by $d$ the Riemannian distance induced by $\mathcal G$. Let $f$ and $g$ be two positive
continuous densities, and let $T$ be the optimal transport map for the
cost $c=\frac{d^2}{2}$ sending $f$ onto $g$. Then there exist
two closed sets $\Sigma_1,\Sigma_2\subset M$ of measure zero, such that $T:M\setminus \Sigma_1 \to M\setminus \Sigma_2$ is
a diffeomorphism of class $W^{1,p}_{loc}$ for any $p\geq 1.$ 
\end{corollary}

In the next result we show that if the cost function is sufficiently close to the ``quadratic'' cost $-x\cdot y$
(recall that this cost is equivalent to $\frac{1}{2}|x-y|^2$),
then the potential is $W^{2,p}$ up to the boundary.
Observe that the smallness parameter $\hat\delta$ is independent of $p$, and that this result is new even in the case $c(x,y)=-x\cdot y$.
 
\begin{theorem}\label{t3}
Suppose $X$ and $Y$ are two $C^2$ uniformly convex bounded domains in $\mathbb{R}^n$. Assume
$f:X\to \R^+$ and $g:Y\to \R^+$ are two continuous positive densities, and
let $u$ be the $c$-convex function associated to the optimal transport problem between $f$ and $g$ with cost 
$c(x,y)$. Suppose $c$ satisfies $(\textbf{C0})$-$(\textbf{C3})$ and 
\begin{equation}\label{902}
\|c+x\cdot y\|_{C^2(X\times Y)} \leq \delta.
\end{equation}
Then there exists $\hat \delta >0$,
depending only on $n$, the modulus of continuity of $f$ and $g$,
and the uniform convexity and $C^2$-smoothness of $X$, and $Y$, such that $u\in W^{2, p}(\overline{X})$ for any $p\geq 1$ provided $\delta \leq \hat \delta$. 
\end{theorem}

The proof of above results is based on the following proposition.

\begin{proposition}\label{p1}
Let $f$ and $g$ be two densities supported in $B_{1/K}\subset\mathcal{C}_1\subset B_K$ and $B_{1/K}\subset \mathcal{C}_2\subset B_K$, respectively. Suppose that $\mathcal{C}_2$ is convex, 
\begin{equation}\label{402} 
\|f-\textbf{1}\|_{L^\infty(\mathcal C_1)}+\|g-\textbf{1}\|_{L^\infty(\mathcal C_2)}\leq \delta,
\end{equation}
and
\begin{equation}\label{403}
\|c(x,y)+x\cdot y\|_{C^2(B_K \times B_K)}\leq \delta.
\end{equation}
Then, for any $p\geq 1$ there exists $\bar \delta >0$, depending only on $n$, $K$, and $p$, such that $u\in W^{2,p}(B_{\frac{1}{4K}})$  provided $\delta \leq \bar \delta$.
\end{proposition}

Note that, in the result above, the smallness of the parameter $\delta$ depends on $p$. So, for the proof of Theorems \ref{t1} and \ref{t3} and Corollary \ref{c1}, it will be crucial to prove that actually $\delta$
can be chosen independently of $p$ (see Lemma \ref{g4}).
Also, as explained in Section \ref{sect:approx} below,
to prove Proposition \ref{p1} we shall first approximate $u$ with smooth solutions and then obtain $W^{2,p}$ a priori estimates that are independent of the regularization. We note that, in this context, such a regularization procedure is nontrivial and require some attention.

\begin{remark}
\label{rmk:other hyp}
As we shall also observe later,
the condition ``$\mathcal{C}_2$ is convex" in Proposition \ref{p1} can be replaced by the assumption
$$\biggl\|u-\frac{1}{2}|x|^2\biggr\|_{L^\infty(B_{\eta_0})}\leq \delta$$ for some fixed $\eta_0\leq 1/K.$
Under this assumption, for any $p\geq 1$ there exists $\bar \delta >0$, depending only on $n$, $K$, $\eta_0$, and $p$, such that $u\in W^{2,p}(B_{\frac{1}{2}\eta_0})$  provided $\delta \leq \bar \delta$.
Moreover, in the above condition, the function $\frac{1}{2}|x|^2$ can be replaced by a $C^2$ convex function $v$
such that $\frac1M{\rm Id}\leq D^2v \leq M{\rm Id}$,
in which case $\bar\delta$ depends also on $M$
and the modulus of continuity of $D^2 v$ .
\end{remark}

\section{Proof of Proposition \ref{p1}}
We begin by observing that, under our assumptions,
it follows by \cite[Theorem 4.3]{DF}
and the argument in the proof of \cite[Theorem 2.1]{CF}
that $u \in C^{1,\alpha}(B_{\frac{1}{2K}})$ for some $\alpha>0$.
Hence, up to replace $K$ by $2K$, we can assume that $u \in C^{1,\alpha}(B_{1/K})$.
In particular it follows by \eqref{eq:u diff} that 
$S(x_0,y_0,u,h)=S(x_0,T_u(x_0),u,h)$, and we can use the notation
$S_h(x_0)=S(x_0,T_u(x_0),u,h)$.

\subsection{Engulfing property of sections}
The first step consists in establish the engulfing property for sections of $u$, which is stated as the following lemma. 

\begin{lemma}[Engulfing property]\label{eng}
There exist universal constants $r_0>0$ and $C>1$ such that,
for $h\leq r_0$  and $x_0\in B_{\frac{2}{3K}},$
$$
x_1\in S_h(x_0)\qquad \Longrightarrow\qquad S_h(x_0)\subset S_{Ch}(x_1).$$ 
\end{lemma}

\begin{proof}
Without loss of generality we may assume $x_0=0, y_0=T_u(x_0)=0$,
and $u(0)=0$. Up to performing the transformations
$$c(x,y)\mapsto \tilde{c}(x,y):=c(x,y)-c(x,0)-c(0,y)+c(0,0),\qquad
u(x)\mapsto \tilde u(x):=u(x)+c(x,0),$$
we may assume 
\begin{equation}\label{a4}
c(x,y)=-x\cdot y+O(|x|^2|y|+|x||y|^2).
\end{equation}
Set $\rho:=\left(\frac{|\mathcal{C}_1|}{|\mathcal{C}_2|}\right)^{1/n}$ so that $|\rho\mathcal C_2|=|\mathcal C_1|$,
and 
let $v$ be a convex function satisfying $(\nabla v)_\sharp\textbf{1}_{\mathcal{C}_1}=\textbf{1}_{\rho\mathcal{C}_2}$ with $v(0)=0$ (we note that $\nabla v$ is the optimal transport map from $\frac{1}{|\mathcal C_1|}\textbf{1}_{\mathcal{C}_1}$
to $\frac{1}{|\mathcal C_2|}\textbf{1}_{\mathcal{C}_2}$
for the quadratic cost). By a compactness argument similar to the proof of \cite[Lemma 4.1]{DF} we have 
\begin{equation}\label{int}
\|u-v\|_{L^\infty(B_{\frac{1}{K}})}\leq \omega(\delta),
\end{equation}
 where
$\omega:\R^+\rightarrow \R^+$ satisfies $\omega(r)\rightarrow 0$ as $r\rightarrow 0.$ Also, since $\rho\mathcal{C}_2$ is convex,
it follows by \cite{C92} that $v$ is smooth
and uniformly convex in $B_{\frac{3}{4K}}.$\\

Thanks to  \eqref{402}, \eqref{403}, and \eqref{int}, we can follow the proof of \cite[Theorem 4.3]{DF} to show that, for small $h$, there exists an affine transform $A$ such that  
\begin{equation}\label{a1}
A(B_{\frac{1}{3}\sqrt{h}})\subset S_h\subset A(B_{3\sqrt{h}}),
\end{equation}
\begin{equation}\label{a2}
A'^{-1}(B_{\frac{1}{3}\sqrt{h}})\subset T_u(S_h)\subset A'^{-1}(B_{3\sqrt{h}})
\end{equation}
and
\begin{equation}\label{a8}
\biggl|u(Ax)-\frac{1}{2}|x|^2\biggr|\leq \eta h \quad \text{in $B_{3\sqrt{h}}$}
\end{equation}
with 
\begin{equation}\label{a3}
\|A\|, \|A^{-1}\|\leq h^{-\theta}, 
\end{equation}
where $\eta, \theta>0$ can be as small as we want, provided $\delta$ is sufficiently small.
Note that \eqref{int} plays the same role as the fact that $u$ is close to a quadratic function, which is used in the proof of \cite[Theorem 4.3]{DF}.
Furthermore, \eqref{a2} and \eqref{a3} imply that $\text{diam}(S_h(x))\leq Ch^{\frac{1}{2}-\theta}.$ 
Hence, if we choose $r_0$ small enough so that $Cr_0^{\frac{1}{2}-\theta}\leq \frac{1}{4K}$,
we see that for $h\leq r_0$ and $x\in B_{\frac{2}{3K}}$ we have $S_h(x)\subset B_{\frac{3}{4K}}.$\\


Now we perform the transformations $c_1(x,y):=c(Ax, A'^{-1}y)$
and $u_1(x):=u(Ax)$, and we use the notation 
$S_h^1=S(0,0, u_1, h)$.
By \eqref{a1} and \eqref{a2} we have
\begin{equation}\label{a5}
B_{\frac{1}{3}\sqrt{h}}\subset S_h^1\subset B_{3\sqrt{h}}
\end{equation}
and
\begin{equation}\label{a6}
B_{\frac{1}{3}\sqrt{h}}\subset T_{u_1}(S_h^1)\subset B_{3\sqrt{h}}.
\end{equation}
Note that 
\begin{equation}
\label{eq:x section}
0\leq u_1(x)+c_1(x,0)-c_1(0,0)-u_1(0)\leq h
\qquad \text{for any $x\in S_h^1$.}
\end{equation} 
Also,  by \eqref{a4} and \eqref{a3} we see that  $\|c_1\|_{C^2(B_{3\sqrt{h}}\times B_{3\sqrt{h}})} \leq C$ for some universal constant $C$. Therefore, thanks to \eqref{a6}
and \eqref{a7},
for any $x,x_1 \in S_h^1$ and $y_1 =T_{u_1}(x_1) \in T_{u_1}(S_h^1)$,
\begin{eqnarray*}
|c_1(x, y_1)-c_1(x_1, y_1) + c_1(x_1, 0) -c_1(x, 0)| \leq \|D_{xy}c_1\|_{C^0(B_{3\sqrt{h}}\times B_{3\sqrt{h}})}|x_1-x|\,|y_1| \leq C_1h
\end{eqnarray*}
for some universal constant $C_1$.
 Hence, by \eqref{eq:x section} applied to both $x$ and $x_1$ we get
\begin{eqnarray*}
u_1(x)+c_1(x,y_1)-c_1(x_1,y_1)-u_1(x_1)
&=&u_1(x)+c_1(x,0)-c_1(0,0)-u_1(0)\\
&&-(u_1(x_1)+c_1(x_1,0)-c_1(0,0)-u_1(0))\\
&&+ c_1(x, y_1) - c_1(x_1, y_1) + c_1(x_1, 0) -c_1(x, 0)\\
&\leq& h+C_1h.
\end{eqnarray*}
Since $x \in S_h^1=S(0,0,u_1,h)$ was arbitrary,
this proves that $S(0,0,u_1,h) \subset S(x_1,y_1,u_1, (1+C_1)h).$
Recalling the relation between $u_1$ and $u$, this proves the desired result. 
\end{proof}

As a consequence of this result, one gets the following:
\begin{corollary}
\label{cor:engulf}
There exists a small constant $r_1$ such that
for $h\leq r_1$  and $x,y\in B_{\frac{1}{2K}}$ the following holds:
suppose $S_h(x)\cap S_t(y)\ne \emptyset,$ $t\leq h.$
Then there exists an universal constant $C'$ such that 
$S_t(y)\subset S_{C'h}(x).$ 
\end{corollary}

\begin{proof}
Fix $z\in S_h(x)\cap S_t(y)$. By Lemma \ref{eng} we have that
$$
S_t(y)\subset S_{Ch}(z)\qquad
\text{and}\qquad x\in S_h(x)\subset S_{Ch}(z)
$$
for some universal constant $C.$
Also, by the argument in the proof of Lemma \ref{eng},
$z \in B_{\frac{2}{3K}}$ for $h$ small enough.
Hence, using Lemma \ref{eng} again we have
$S_{Ch}(z)\subset S_{C^2h}(x),$
thus $S_t(y)\subset S_{C'h}(x)$ with $C':=C^2.$
\end{proof}

It is well known that the property of sections 
stated in Corollary \ref{cor:engulf} implies the following Vitali covering lemma (see for instance \cite[Lemma 4.6.2]{Fbook} for a proof):
\begin{lemma}[Vitali covering]\label{vital}
Under the assumptions of Proposition \ref{p1},  let $D$ be a compact subset of $B_{\frac{1}{2K}},$
and let $\{S_{h_x}(x)\}_{x\in D}$ be a family of sections with $h_x\leq r_1.$ Then, there exists a finite number of 
sections $\{S_{h_{x_i}}(x_i)\}_{i=1,\ldots, m}$ such that 
$$D\subset \bigcup_{i=1}^mS_{h_{x_i}}(x_i)$$
with $\{S_{\sigma h_{x_i}}(x_i)\}_{i=1,\ldots, m}$ disjoint, where $\sigma>0$ is a universal constant.
\end{lemma}

\subsection{Approximation argument}
\label{sect:approx}

In the next sections we will prove our $W^{2,p}$ estimates by controlling the measure of the super-level
sets of the Hessian of $u$. Because we shall need to use the pointwise value of $D^2u$,
we need an approximation argument in order to work with $C^2$ convex functions.
Since in this setting this is not a standard procedure, we now provide the details.

Given $u$ as in Proposition \ref{p1},
we set $\rho:=\left(\frac{|\mathcal{C}_1|}{|\mathcal{C}_2|}\right)^{1/n}$ so that $|\rho\mathcal C_2|=|\mathcal C_1|$,
and 
let $v$ be a convex function satisfying $(\nabla v)_\sharp\textbf{1}_{\mathcal{C}_1}=\textbf{1}_{\rho\mathcal{C}_2}$ with $v(0)=0$. Since 
$$
\|u-v\|_{L^\infty(B_{\frac{1}{K}})} \to 0\qquad \text{as $\delta \to 0$}
$$
(see \eqref{int}), 
as in the proof of Lemma \ref{eng}
we can choose $\delta$ small enough so that,
 for any $x \in B_{\frac{1}{2K}}$ and $h>0$ small but universal,
 the section $S_h(x)$ satisfies
 \eqref{a1}, \eqref{a2}, \eqref{a8}, and \eqref{a3}.
 
 We now consider $f_\epsilon:\mathcal C_1\to \R$ and $g_\epsilon:\mathcal C_2\to \R$ as sequence 
 of $C^\infty$ densities that approximate $f$ and $g$ respectively,
 and denote by $u_\epsilon$ the potential function
 for the optimal transport problem from $f_\epsilon$ to $g_\epsilon$ with cost $c$. Without loss of generality, we can assume that $u_\epsilon(0)=u(0)$.
 
 Then, by a compactness argument it follows that 
 $$
\|u_\epsilon-u\|_{L^\infty(B_{\frac{1}{K}})} \to 0\qquad \text{as $\epsilon \to 0$}
$$
Since $u$ is strictly convex,
choosing $\epsilon$ sufficiently small
we see that the sections $S_h^\epsilon(x)=S(x,T_{u_\epsilon}(x),u_\epsilon,h)$ satisfy 
\eqref{a1}, \eqref{a2}, \eqref{a8}, and \eqref{a3} with bounds independent of $\epsilon$.

In particular, assuming $\delta$ is small enough,
by \cite[Theorem 4.3]{DF}
applied to $\frac{1}{h}u_\epsilon(A\sqrt{h}x)$ we deduce that ${u}_\epsilon$ is of class $C^{1, 6/7}$ in $A(B_{\frac{1}{4}\sqrt{h}})$.
By duality, similarly we also have that its $c$-transform ${u}_\epsilon^{c}$ is of class $C^{1, 6/7}$ inside $A'^{-1}(B_{\frac{1}{4}\sqrt{h}}).$ Hence, by \cite[Theorem 2.3]{CF1} we deduce that ${u}_\epsilon$ is of class $C^2$ in a neighborhood of $x$.
Since $x \in B_{\frac{1}{2K}}$ was arbitrary, this proves that $u_\epsilon \in C^2(B_{\frac{1}{2K}})$ for any $\epsilon >0$ small enough.

Hence, up to proving our $W^{2,p}$ estimates with $u_\epsilon$
in place of $u$ and then letting $\epsilon \to 0$,
in the next sections we shall directly assume that $u \in C^2$.

\subsection{Density estimates}
The goal here is to show that, given a section
$S_h(x)\subset B_{\frac{1}{2K}},$ the density of ``bad points'' where the Hessian of $u$ is large has measure
that goes to zero as $\delta\to 0.$

Fix $x_0\in B_{\frac{1}{2K}},$ and let $y_0=T_u(x_0).$ Without loss of generality, we may assume $x_0=y_0=0.$ Also, as in 
the proof of Lemma \ref{eng} we can assume that \eqref{a4} holds.
In this way it follows that, for $h$ small, \eqref{a1}, \eqref{a2}, \eqref{a8}, and \eqref{a3} hold.

Perform the transformations
$$c(x,y)\mapsto \bar{c}(x,y):=\frac{1}{h}c(\sqrt{h}Ax, \sqrt{h}A'^{-1}y);$$
$$u(x)\mapsto \bar{u}(x):=\frac{1}{h}u(\sqrt{h}Ax);$$
$$f(x)\mapsto\bar{f}(x):=f(\sqrt{h}Ax),\qquad g(y)\mapsto \bar{g}(y)=g(\sqrt{h}A'^{-1}y).$$
Note that, by \eqref{a4} and \eqref{a3}, we have
\begin{equation}\label{a7}
\|\bar{c}+x\cdot y\|_{C^2(B_8\times B_8)}\leq \delta
\end{equation}
provided $h$ is sufficiently small.
Also, it follows by \eqref{a1}, \eqref{a2}, and \eqref{a8} that
\begin{equation}\label{a9}
B_{\frac{1}{3}}\subset S(0,0, \bar{u}, 1)\subset B_{3},
\end{equation}
\begin{equation}\label{b1}
B_{\frac{1}{3}}\subset T_{\bar{u}}(S(0,0, \bar{u}, 1))\subset B_{3},
\end{equation}
and
\begin{equation}\label{b2}
\biggl\|\bar{u}(x)-\frac{1}{2}|x|^2\biggr\|_{L^\infty(B_3)}\leq \eta.
\end{equation}
We now construct a smooth function $w$ that well approximates $\bar u$. Denote $X_1:=S(0,0, \bar{u}, h)$ and $Y_1:=T_{\bar{u}}(S(0,0, \bar{u}, h)).$

\begin{lemma}
Set $\rho:=\left(\frac{|X_1|}{|Y_1|}\right)^{1/n},$
and let $w$ be a convex function such that
 $(\nabla w)_\sharp \textbf{1}_{X_1}=\textbf{1}_{\rho Y_1}$ and $w(0)=u(0).$ 
 Then, for any $\gamma>0,$ there exist $\delta_\gamma, \eta_\gamma>0$ such that 
 \begin{equation}\label{b3}
 \|\bar{u}-w\|_{L^\infty(B_{1/4})}\leq \gamma
 \end{equation}
  and
  \begin{equation}\label{b4}
  \|w\|_{C^3(B_{1/6})}\leq C
  \end{equation}
   provided 
 $\delta\leq \delta_\gamma$ and $\eta\leq \eta_\gamma,$ where $C$ is a universal constant.
\end{lemma}

\begin{proof}
The bound
\eqref{b3} follows from a compactness argument similar to the proof of \cite[Lemma 4.1]{DF}.
Also, taking $\gamma\leq \eta$, \eqref{b2} and \eqref{b3} imply that
\begin{equation}\label{b5}
\biggl\|w(x)-\frac{1}{2}|x|^2\biggr\|_{L^\infty(B_{1/4})}\leq 2\eta.
\end{equation}
Thanks to \eqref{b5}, as in the proof of Lemma \ref{eng} (see also
Step 1 in the proof of \cite[Theorem 4.3]{DF}) we can apply \cite{C92}  
to deduce that $ \|w\|_{C^3(B_{1/6})}\leq C$ for
some universal constant $C.$
\end{proof}

Let $L$ be the operator defined by $$L\U(x):=D^2\bar{u}(x)+D_{xx}\bar{c}\bigl(x, T_{\U}(x)\bigr).$$
By \eqref{402} and \eqref{a7}, we have
\begin{equation}
\label{MA }
\det (L\U(x))=\left|\det\left(D_{xy}\bar{c}\bigl(x,T_{\U}(x)\bigr) \right) \right| \frac{\bar{f}(x)}{\bar{g}(T_{\U}(x))}=1+O(\delta).
\end{equation}
We now follow the argument in \cite{C902} to establish the density estimate. Since the argument is rather standard,
we shall just emphasize the main points, referring to \cite{C902} or
\cite[Chapter 4.7]{Fbook} for more details.

\begin{lemma}\label{den1}
Let $\U, w$ be as above,
and denote by $\Gamma\left(\U-\frac{w}{2}\right)$ the convex envelope of $\U-\frac{w}{2}.$
Then, for any Borel set $E\subset B_{1/6},$ we have 
\begin{equation}
\biggl|\nabla \Gamma\Bigl(\U-\frac{w}{2}\Bigr)(E)\biggr|\leq \biggl(\frac{1}{2^n}+O(\delta)\biggr) \biggl|E\cap\biggl\{\Gamma\Bigl(\U-\frac{w}{2}\Bigr)=\U-\frac{w}{2}\biggl\}\biggl|
\end{equation}
\end{lemma}

\begin{proof}
Noticing that
$\det D^2w= 1$,
$\det D^2\bar u=1+O(\delta)$, and
$D_{xx}\bar{c}=O(\delta)$,
since $w$ is uniformly convex
and $\det D^2\Gamma(\U-\frac{w}{2})$ is a measure supported on $\{\Gamma(\U-\frac{w}{2})=\U-\frac{w}{2}\}$,
it follows by the Area Formula (see for instance \cite[Proposition A.4.19]{Fbook}) that
\begin{eqnarray*}
\biggl|\nabla \Gamma\Bigl(\U-\frac{w}{2}\Bigr)(E)\biggr|&=&\int_{E\cap\{\Gamma(\U-\frac{w}{2})=\U-\frac{w}{2}\}}\det D^2\Bigl(\U-\frac{w}{2}\Bigr)\\
&=&\int_{E\cap\{\Gamma(\U-\frac{w}{2})=\U-\frac{w}{2}\}} \det\biggl[ L\U-\Bigl(D^2\frac{w}{2}+D_{xx}\bar{c}\bigl(x, T_{\U}x\bigr)\Bigr)\biggr]\\
&\leq& \int_{E\cap\{\Gamma(\U-\frac{w}{2})=\U-\frac{w}{2}\}}\biggl(\det(L\U)^{1/n}-  \det\Bigl[\Bigl(D^2\frac{w}{2}+O(\delta)\Bigr)\Bigr]^{1/n}\biggr)^n\\
&\leq&  \int_{E\cap\{\Gamma(\U-\frac{w}{2})=\U-\frac{w}{2}\}} \biggl(1+O(\delta)-\Bigl(\frac{1}{2}(\det D^2w)^{1/n}-O(\delta)\Bigr)\biggr)^n\\
&\leq& \biggl(\frac{1}{2}+O(\delta)\biggr)^n\biggl|E\cap\biggl\{\Gamma\Bigl(\U-\frac{w}{2}\Bigr)=\U-\frac{w}{2}\biggr\}\biggr|,
\end{eqnarray*}
where we used the inequality
$$
[\det (A+B)]^{1/n}\geq (\det A)^{1/n}+(\det B)^{1/n}\qquad \forall\,A,B \text{ symmetric, nonnegative definite}
$$
(see for instance \cite[Lemma A.1.3]{Fbook} for a proof).
\end{proof}

By using Lemma \ref{den1}, we can follow the lines of proof of \cite[Lemma 6]{C902} (see also the proof of \cite[Lemma 4.7.1]{Fbook})
to establish the estimate
$$
\frac{|\{\Gamma(\U-\frac{w}{2})=\U-\frac{w}{2}\}\cap B_{1/8}|}{|B_{1/8}|}\geq 1-C\delta^{1/2},
$$
from which one immediately obtain the following bound 
(see \cite[Corollary 1]{C902} or Step 6 in the proof of 
\cite[Lemma 4.7.1]{Fbook}):

\begin{lemma}[Density estimate]\label{den3}
Let $\U$ be as above. Then there exist universal constants $N>1, \eta>0$ such that
\begin{equation}\label{den4}
\left|\left\{x\in S^{\U}_{\eta}(0): \|D^2\U(x)\|\geq N\right\}\right|\leq N\delta^{1/2}\left|S^{\U}_{\eta}(0)\right|.
\end{equation}
\end{lemma}

\subsection{$W^{2,p}$ estimate}
We now prove our $W^{2,p}$ interior estimates.
Recall that we are assuming that $u \in C^2$.

As in the proof of Lemma \ref{eng}, for any $x\in B_{\frac{1}{2K}}$ and $h>0$ small enough, there exists an affine transformation $A$ with $\det A=1$ such that 
\begin{equation}\label{f1}
A(B_{\frac{1}{3}\sqrt{h}})\subset S_h(x)\subset A(B_{3\sqrt{h}}).
\end{equation}
We define the normalized size of the section $S_h(x)$ as
\begin{equation}\label{f2}
\mathbf{a}(S_h(x)):=\|A^{-1}\|^{2}.
\end{equation}
Although $A$ is not unique, if $A_1$ and $A_2$ are two affine transformations that satisfy \eqref{f1} then both $\|A_1^{-1}A_2\|$ and $\|A_2^{-1}A_1\|$ are universally bounded, thus the normalized size is well defined up to universal constants.

With the notation from the previous section, we see that the estimate \eqref{den4} can be rewritten in terms of $u$ and becomes 
\begin{equation}\label{f3}
\bigl|S_h(x)\cap \bigl\{\|D^2u(x)\|\geq N\mathbf{a}\bigl(S_h(x)\bigr)\bigr\}\bigr|\leq C\delta^{1/2}|S_{h}(x)|
\end{equation}
for any $h$ small enough.
Also, since $\det D^2 \bar u=1+O(\delta)$, it follows that 
$$
\|D^2u\|\leq N\qquad \Longrightarrow \qquad D^2 u \geq \frac{1}{2N^{n-1}}{\rm Id}.
$$
Thus, up to enlarging $N$ and using Lemma \ref{den3} again, we deduce that
$$ |S_{h}(u)|\leq C\biggl|S_{\sigma h}\cap\biggl\{\frac{\mathbf{a}\bigl(S_h(x)\bigr)}{N}\leq \|D^2u\|\leq N\mathbf{a}\bigl(S_h(x)\bigr)\biggr\}\biggr|,$$
that combined with \eqref{f3} yields
\begin{equation}\label{f4}
\bigl|S_h(x)\cap \bigl\{\|D^2u\|\geq N\mathbf{a}\bigl(S_h(x)\bigr)\bigr\}\bigr|\leq C\delta^{1/2}\biggl|S_{\sigma h}\cap\biggl\{\frac{\mathbf{a}\bigl(S_h(x)\bigr)}{N}\leq \|D^2u\|\leq N\mathbf{a}\bigl(S_h(x)\bigr)\biggl\}\biggr|.
\end{equation}
Also, by \eqref{a3} we have
\begin{equation}\label{f5}
\text{diam}(S_h(x))\leq Ch^{1/2}\|A\|\leq Ch^{1/2-\theta}\leq \hat C\mathbf{a}\bigl(S_h(x)\bigr)^{-\beta},
\end{equation}
where $\beta:=\frac{1}{4\theta}-\frac{1}{2}.$ 

Let $M\gg 1$ to be fixed later, set $\rho_0:=\frac{1}{2K},$ and for $m\geq 1$ we define $\rho_m$ inductively by
\begin{equation}
\label{eq:rho}
\rho_m:=\rho_{m-1}-\hat CM^{-m\beta},
\end{equation}
where the constants $\hat C, \beta$ are as those in \eqref{f5}.
Note that, by taking $M$ large enough so that $$\sum_{m=1}^\infty \hat CM^{-m\beta}<\frac{1}{4K},$$ we can ensure that 
$\rho_m\geq \frac{1}{4K}$ for all $m\geq 1.$

Now, for $k \geq 0$ we set $D_k:=\{x\in B_{\rho_k}: \|D^2u\|\geq M^k\}.$ We shall prove the following lemma.
\begin{lemma}\label{f7}
$|D_{k+1}|\leq N\delta^{1/2}|D_k|.$
\end{lemma}
\begin{proof}
Let $M\gg N$ to be chosen later, and
for any $x\in D_{k+1}$ choose a section $S_{h_x}(x)$ such that
\begin{equation}\label{f6}
 \textbf{a}(S_{h_x}(x))=NM^k.
 \end{equation}
 Such a section always exists because
 $\textbf{a}(S_{h})\approx 1<NM^k$ when $h=h_0$
 is a small but fixed universal constant, while  
 $$\textbf{a}(S_{h})\approx \|D^2u(x)\|\geq M^{k+1}>NM^k \qquad \text{as $h\rightarrow 0$}
 $$
 (the estimate $\textbf{a}(S_{h})\approx \|D^2u(x)\|$ follows by a simple Taylor expansion, see for instance \cite[Remark 4.7.5]{Fbook}).
 Hence, by continuity there exists $h_x\in (0, h_0)$ such that \eqref{f6} holds.

 Now, by Lemma \ref{vital}, we can find a finite number of sections $\{S_{h_{x_i}}(x_i)\}_{i=1,\ldots,m}$
 covering $D_{k+1}$ such that $\{S_{\sigma h_{x_i}}(x_i)\}_{i=1,\ldots, m}$ are disjoint. Then, it follows by \eqref{f4} that
 \begin{equation}\label{f8}
\bigl|S_{h_i}(x_i)\cap \bigl\{\|D^2u\|\geq N^2M^k\bigr\}\bigr|\leq N\delta^{1/2}\bigl|S_{\sigma h_i}(x_i)\cap\bigl\{M^k\leq \|D^2u\|\leq N^2M^k\bigr\}\bigr|.
\end{equation}
Hence, recalling \eqref{f5} and \eqref{eq:rho}, we obtain
\begin{eqnarray*}\label{f9}
|D_{k+1}|&\leq& \sum_{i=1}^m \bigl|S_{h_i}(x_i)\cap \bigl\{\|D^2u\|\geq N^2M^k\bigr\}\bigr|\\
&\leq&N\delta^{1/2}\sum_{i=1}^m\bigl|S_{\sigma h_i}(x_i)\cap\bigl\{M^k\leq \|D^2u\|\leq N^2M^k\bigr\}\bigr|\\
&\leq&N\delta^{1/2}|D_k|
\end{eqnarray*}
provided $M\geq N^2.$ 
\end{proof}

\begin{proof}[Proof of Proposition \ref{p1}]
Thanks to Lemma \ref{f7}, we have
$$|D_k|\leq (N\delta^{1/2})^k|D_0|\leq \frac{1}{M^{k(p+1)}}|B_{\frac{1}{2K}}|$$
provided $\delta\leq \frac{1}{N^2M^{2(p+1)}}.$
Therefore
\begin{eqnarray*}
\int_{B_{\frac{1}{4K}}}\|D^2u\|^p&=&p\int_{B_{\frac{1}{4K}}}t^{p-1}|B_{\frac{1}{4K}}\cap\{\|D^2u\|\geq t\}|\\
&\leq&C\sum_{k=1}^\infty M^{kp}|D_k|\leq C,
\end{eqnarray*}
as desired.

\end{proof}

\section{Proof of Theorem \ref{t1} and Corollary \ref{c1}}

\subsection{Proof of Theorem \ref{t1}.} 
By the argument in \cite[Section 3]{DF}, we only need to establish the following result, which is a strengthened version of
Proposition \ref{p1} for continuous densities. Indeed, the lemma shows that the exponent $p$ in the $W^{2,p}$ estimate is {\em independent} of
the parameter $\delta.$ This is crucial in showing that the singular set $\Sigma$ can be chosen independently of $p.$
\begin{lemma}\label{g4}
Let $f, g$ be two {continuous} densities supported in $B_{1/K}\subset X_1\subset B_K$ and $B_{1/K}\subset Y_1\subset B_K$ respectively. Suppose that
\begin{equation}\label{g1} 
\|f-\textbf{1}\|_{L^\infty(X_1)}+\|g-\textbf{1}\|_{L^\infty(Y_1)}\leq \delta,
\end{equation}
\begin{equation}\label{g2}
\biggl\|u-\frac{1}{2}|x|^2\biggr\|_{L^\infty(B_K)}\leq \delta
\end{equation}
and 
\begin{equation}\label{g3}
\|c(x,y)+x\cdot y\|_{C^2(B_K\times B_K})\leq \delta.
\end{equation}
Then there exists $\bar \delta >0$, depending only on $n$ and $K$, such that $u\in W^{2,p}(B_{\frac{1}{2K}})$ for any $p\geq1$  provided $\delta \leq \bar \delta$.
\end{lemma}

\begin{proof}

Fix $x_0\in B_{\frac{1}{2K}},$ and without loss of generality assume $x_0=0,$ $T_u(x_0)=0$, and $u(x_0)=0$.
For small $h,$ similarly to the proof of Lemma \ref{eng}, there exists an affine transformation $A$ with 
$\det A=1, \|A\|, \|A^{-1}\|\leq h^{-\theta},$ such that \eqref{a1}, \eqref{a2}, and \eqref{a3} hold, where 
$\theta$ can be as small as we want provided $\delta$ is sufficiently small. Also, we may assume \eqref{a4} holds.

Given a set $E$, let $[E]$ denote its convex hull.
By \cite[Lemma 3.2]{CF} we have that 
\begin{equation}\label{h1}
{\rm dist}(S_h, [S_h])\leq Ch^{1-6\theta}.
\end{equation}
Also, by $C^{1,\alpha}$ regularity of $u$ (hence, $C^{0,\alpha}$
regularity of $T_u$), we have 
\begin{equation}\label{h2}
{\rm dist}\bigl(T_u(S_h), T_u([S_h])\bigr)\leq Ch^{(1-6\theta)\alpha}.
\end{equation}

Perform the transformations

$$
u(x)\mapsto \frac{1}{h}u(\sqrt{h}A^{-1}x):=u_1(x);
$$
$$
c(x, y) \mapsto \frac{1}{h}c(\sqrt{h}A^{-1}x, \sqrt{t}A'y):=c_1(x,y);
$$
$$
f(x) \mapsto f_1(x):=f(\sqrt{h}A^{-1}x),\
\qquad g(y)\mapsto g_1(y):=g(\sqrt{h}A'y);
$$
$$
S_h\mapsto \tilde{S}:=\frac{1}{\sqrt{h}}A (S_h).
$$
Also, set $\mathcal{C}_1:=[\tilde{S}],$  $\mathcal{C}_2:=T_{u_1}([\tilde{S}])$, 
 $\bar{f}:=f_1\textbf{1}_{\mathcal{C}_1},$ 
 and $\bar{g}:=g_1\textbf{1}_{\mathcal{C}_2}.$

By \eqref{a1}, \eqref{a2}, \eqref{h1}, \eqref{h2} we have
$$
B_{\frac{1}{4}}\subset \mathcal{C}_1\subset B_{4};
$$
$$
B_{\frac{1}{4}}\subset \mathcal{C}_2\subset B_{4};
$$
$$
\|\bar{f}-\textbf{1}_{\mathcal{C}_1}\|_{L^\infty(B_{4})}=o(1),\quad \|\bar{g}-\textbf{1}_{\mathcal{C}_2}\|_{L^\infty(B_{4})}=o(1) \rightarrow 0\ \ \  \text{as}\ h\rightarrow 0.
$$
It is also easy to check that
$$
\|c_1+x\cdot y\|_{C^2(B_{4}\times B_{4})}=o(1) \rightarrow 0\ \ \ \text{as}\ h \rightarrow 0,
$$

Since $\mathcal{C}_1$ is convex, we can apply Proposition \ref{p1} (switch the role of $x$ and $y$) to deduce that, given any $p \geq 1$, we can choose $h$ small enough so that $u^c_1$, the $c$-transform of $u_1$, belongs to
$W^{2,p}(B_{\frac{1}{8}})$ provided $h$ is sufficiently small. By a symmetric argument (or using that $D^2 u$ and $D^2 u^c$
are related), one gets that, given any $p\geq 1,$
$u_1\in W^{2,p}(B_{\frac{1}{8} })$ provided $h$ is sufficiently small.
Rescaling back to $u$
this proves that, given $p\geq 1,$ $u\in W^{2,p}(B_r)$ provided $r$ is small enough (the smallness depending on $h$).
Thanks to this fact, Lemma \ref{g4} follows from a standard covering argument.

\end{proof}

\begin{proof}[Proof of Theorem \ref{t1}] 
Theorem \ref{t1} is an easy consequence of Lemma \ref{g4}, following the argument in \cite[Section 3]{DF}.
\end{proof}
 
\begin{proof}[Proof of Corollary \ref{c1}] 
 Corollary \ref{c1} follows by the same reasoning as the proof of \cite[Theorem 1.4]{DF}.
 \end{proof}


\section{Proof of Theorem \ref{t3}}
Since interior $W^{2,p}$ estimates follows from Lemma \ref{g4} and \cite[Lemma 3.11]{CF1},
we focus on the estimate near the boundary.

Under the assumptions of Theorem \ref{t3}, it is proved in \cite{Chen} that, for 
any $\alpha<1$, there exists $\bar{\delta}>0$ such that
$u\in C^{1,\alpha}(\bar{X})$
provided $\delta\leq \bar{\delta}$. 
Let
$$\bar{h}(x):=\max\{h>0: S_h(x)\subset X\},$$
and set $S_x:=S_{\bar h(x)}(x)$.
As in the proof of Lemma \ref{eng}, there exists an affine transformation $A$ with 
$\det A=1, \|A\|, \|A^{-1}\|\leq \bar h(x)^{-\theta},$ such that
\begin{equation}\label{161}
B_{\frac{1}{3}\sqrt{\bar{h}(x)}} \subset A(S_x)\subset B_{3\sqrt{\bar{h}(x)}}.
\end{equation}
Hence, since $\|A\|, \|A^{-1}\|\leq \bar h(x)^{-\theta}$,
it follows by \eqref{161} and the definition of $\bar{h}(x)$ that $$\text{dist}(x, \partial X)\leq C\bar{h}(x)^{\frac{1}{2}-\theta},$$
which proves that
\begin{equation}\label{162}
S_{x}\subset X_{C\bar{h}(x)^{\frac{1}{2}-\theta}}:=\left\{z \in X: \text{dist}(z, \partial X)\leq C\bar{h}(x)^{\frac{1}{2}-\theta}\right\}.
\end{equation}

Fix $h_0>0$ small but universal. Similarly to the proof of Lemma 3.1, we can find a Vitali covering of $X_{h_0},$  denoted by $\{S_{\bar{h}(x_i)}(x_i)\},$ such that the sections
$\{S_{\sigma\bar{h}(x_i)}(x_i)\}$ are disjoint.

Now, fix $x_0\in X_{h_0}$ a point close to $\partial X.$ Without loss of generality we may assume $x_0=0,$ $T_u(x_0)=0,$ and
$u(x_0)=0$.
Consider the section $S_{\bar{h}}:=S(0, 0, u, \bar h(0)).$
As in the proof Lemma \ref{g4}, we perform the transformations \eqref{f1}, \eqref{f2}, \eqref{f3}, and \eqref{f4},
and we set $\mathcal{C}_1:=[\tilde{S}]$, $\mathcal{C}_2:=T_{u_1}([\tilde{S}]),$ 
$\bar{f}:=f_1\textbf{1}_{\mathcal{C}_1},$ 
and $\bar{g}:=g_1\textbf{1}_{\mathcal{C}_2},$ so that
$$
B_{\frac{1}{4}}\subset \mathcal{C}_1\subset B_{4};
$$
$$
B_{\frac{1}{4}}\subset \mathcal{C}_2\subset B_{4};
$$
$$
\|\bar{f}-\textbf{1}_{\mathcal{C}_1}\|_{L^\infty(B_{4})}=o(1),\quad \|\bar{g}-\textbf{1}_{\mathcal{C}_2}\|_{L^\infty(B_{4})}=o(1) \rightarrow 0\ \ \  \text{as}\ \bar{h}, \delta\rightarrow 0;
$$
$$
\|c_1+x\cdot y\|_{C^2(B_{4}\times B_{4})}=o(1) \rightarrow 0\ \ \ \text{as}\ \bar{h}, \delta \rightarrow 0.
$$
Note that, by \eqref{a8}, $u_1$ is arbitrarily close to the function $\frac{1}{2}|x|^2$.
Let $v$ be the convex function solving $(\nabla v)_\sharp \textbf{1}_{\mathcal{C}_1}=\textbf{1}_{\rho\mathcal{C}_2}$ with $v(0)=u(0)$ and $\rho:=\left(\frac{|\mathcal{C}_1|}{|\mathcal{C}_2|}\right)^{1/n}.$ 
By a compactness argument we have that 
$$
\|u_1-v\|_{L^\infty(B_{\frac{1}{4}})}\leq \omega(\delta),
$$
where $\omega:\R^+\rightarrow \R^+$ satisfies $\omega(r)\rightarrow 0$ as $r\rightarrow 0.$ 
This implies that also $v$ is uniformly close to the function $\frac{1}{2}|x|^2$ inside $B_{\frac{1}{4}}$, hence \cite{C92}
yields that $\|v\|_{C^3(B_{\frac{1}{5}})}\leq C$ for some universal constant $C$,
and that $v$ is uniformly convex in $B_{\frac{1}{5}}$. Thus,
if we set $S_t^1:=S(0,0,u_1,t),$ we can apply Proposition \ref{p1} to deduce that
$\|u_1\|_{W^{2,p}(S^{u_1}_t)}\leq C$ for some universal constants $t, C,$.

 Rescaling back to $u,$ this proves that
\begin{equation}\label{g1 p}
\int_{S_{t\bar{h}(x_0)}} \|D^2u\|^p\leq C\bar{h}(x_0)^{-2p\theta}|S_{\sigma\bar{h}(x_0)}|.
\end{equation}

Now, consider the family of sections $\mathcal{F}_k:=\{S_{\bar{h}(x_i)}(x_i):h_02^{-k-1}\leq \bar{h}(x_i)\leq h_02^{-k}\}.$
Then, since $|X_{r}|\approx r$ for $r$ small
and the sections $\{S_{\sigma\bar{h}(x_i)}(x_i)\}$ are disjoint,
it follows by \eqref{g1 p} and \eqref{162} that
\begin{eqnarray*}
\sum_{S_{\bar{h}(x_i)}(x_i)\in \mathcal{F}_k}\int_{S_{\bar{h}(x_i)}(x_i)}\|D^2u\|^p&\leq& C\sum_{S_{\bar{h}(x_i)}(x_i)\in \mathcal{F}_k}\bar{h}(x_i)^{-2p\theta}|S_{\sigma\bar{h}(x_i)}|\\
&\leq& C2^{2kp\theta} |X_{C(h_02^{-k})^{\frac{1}{2}-\theta}}|\\
&\leq&  C2^{2kp\theta}(h_02^{-k})^{\frac{1}{2}-\theta}\\
&\leq& C2^{-k(\frac{1}{2}-3p\theta)}
\end{eqnarray*}
Choosing $\theta$ small enough so that $3p\theta\leq \frac{1}{4}$,
we can sum the above estimate with respect to $k$
to get
$$
\int_{X_{h_0}}\|D^2u\|^p \leq C.
$$
Since $\int_{X\setminus X_{h_0}}\|D^2u\|^p \leq C$
by interior regularity, this concludes the proof. \qed




\bibliographystyle{amsplain}

\end{document}